\documentclass[a4paper,12pt]{article}
\usepackage[utf8]{inputenc}
\usepackage{amsmath,amssymb,amsthm,color}
\newcommand{\norm}[1]{\left\Vert#1\right\Vert}
\newcommand{\abs}[1]{\left\vert#1\right\vert}

\newcommand{\lie}[1]{\mathfrak{#1}}     

\newcommand{\Span}[1]{\operatorname{Span}\left\{#1\right\}}

\newcommand{\hook}{\lrcorner\,}

\newcommand{\so}{\mathfrak{so}}

\DeclareMathOperator{\Tr}{tr}

\DeclareMathOperator{\ric}{Ric}

\newcommand{\dfn}[1]{\emph{#1}}
\newcommand{\id}{\mathrm{Id}}   

\DeclareMathOperator{\End}{End}

\newcommand{\D}[1]{\frac{\partial}{\partial #1}}

\theoremstyle{plain}
\newtheorem{theorem}{Theorem}[section]
\newtheorem{proposition}[theorem]{Proposition}
\newtheorem{lemma}[theorem]{Lemma}
\newtheorem{corollary}[theorem]{Corollary}

\newtheorem{remark}[theorem]{Remark}

\author{Diego Conti and Marisa Fern\'andez}
\title{Einstein almost cok\"ahler manifolds}
\begin{document}

\maketitle

\begin{abstract}
We study an odd-dimensional analogue of the Goldberg conjecture
for compact Einstein almost K\"ahler manifolds. We give an explicit non-compact example of an Einstein almost cok\"ahler manifold that is not cok\"ahler.
We prove that compact Einstein almost cok\"ahler manifolds with non-negative $*$-scalar curvature
are cok\"ahler (indeed, transversely Calabi-Yau); more generally, we 
give a lower and upper bound for the $*$-scalar curvature in the case that the structure is not cok\"ahler. We prove  similar bounds for almost K\"ahler Einstein manifolds that are not K\"ahler.
\end{abstract}

\section{Introduction}
An \emph{almost contact metric structure} $(\alpha, \omega, g)$ on a $(2n+1)$-dimensional 
differentiable manifold $M$ is determined by a pair $(\alpha, \omega)$ of differential forms, 
where $\alpha$ is a $1$-form and $\omega$ is a $2$-form on $M$, and a 
Riemannian metric $g$ on $M$ such that each point of $M$ has an orthonormal coframe $\{e^1,\dotsc, e^{2n+1}\}$ with
\[\alpha=e^{2n+1}, \qquad \omega=e^1\wedge e^2+e^3\wedge e^4+\dots + e^{2n-1}\wedge e^{2n}.\]
If in addition, $\alpha$ and $\omega$ are both parallel with respect to the Levi-Civita connection of the metric $g$, then
$(\alpha, \omega, g)$ is called a \emph{cok\"ahler structure}, and $(M,\alpha, \omega, g)$
is called a \emph{cok\"ahler manifold} \cite{Li}.

By analogy with the terminology used in almost Hermitian geometry (see \cite{Goldberg, Gray, Koto}), we say that
an almost contact metric structure $(\alpha, \omega, g)$ on a manifold $M$ is \dfn{almost cok\"ahler} if 
$\alpha$ and $\omega$ are both closed.
We call $(M,\alpha, \omega, g)$  an \dfn{almost cok\"ahler manifold}.  
Then, the Riemannian product $M\times\mathbb{R}$ (or $M\times S^1$) is an almost K\"ahler manifold
(in particular, K\"ahler if $(\alpha, \omega, g)$ is a cok\"ahler structure) and
$(M,\alpha, \omega)$ is a cosymplectic manifold in the sense of Libermann  \cite{Libermann} since 
$\alpha\wedge\omega^n$ is a volume form of $M$.

In the last years, the geometry and topology of cok\"ahler and almost cok\"ahler  manifolds
have been studied by several authors (see for example \cite{BFM, BG, BO, Blair2, CLM, de Nicola, GMP, Li} and the references therein).

Concerning the geometry of compact almost K\"ahler manifolds, 
the Goldberg conjecture states that the almost complex structure
of a compact Einstein almost K\"ahler manifold is integrable \cite{Goldberg}.
In \cite{Sekigawa}, Sekigawa gives a proof of this conjecture
under the assumption that the scalar curvature of the almost K\"ahler manifold is non-negative. 
This assumption can be replaced by the condition that the $*$-scalar curvature be positive (Corollary~\ref{cor:evengoldberg}); more generally, the same type of argument leads to an estimate for the $*$-scalar curvature (Theorem \ref{thm:evenpinch}).
On the negative side, a complete, almost K\"ahler Einstein manifold  which is not K\"ahler 
was constructed in \cite{ADM} (see also \cite{Hirobe}); this example is not compact, and its scalar curvature is negative.  

An odd-dimensional analogue of the Goldberg conjecture was considered in \cite{Boyer-Galicki:Einstein}, where it is proved that a compact $K$-contact Einstein manifold is Sasakian (see also  \cite{ADM:odd}). Following \cite{CP},
in this paper we consider another odd-dimensional version of this problem, namely:

\medskip
{\em Are all compact Einstein almost cok\"ahler manifolds cok\"ahler?}

\medskip

We note that a negative answer  would disprove the Goldberg conjecture proper, as the product of an Einstein, strictly almost cok\"ahler manifold with itself is Einstein and strictly almost K\"ahler (Proposition~\ref{prop:product}).

A key tool to  attack this problem is the Weitzenb\"ock formula applied to the harmonic forms $\alpha$ and $\omega$ (Lemma \ref{lemma:nablastarnabla}). Indeed, this formula implies that Einstein cok\"ahler manifolds, unlike their even-dimen\-sional counterpart, are Ricci-flat (Proposition~\ref{prop:nocokahlereinstein}). In addition, it implies that any Einstein almost cok\"ahler manifold has non-positive scalar curvature.

A second ingredient is an equality taken from \cite{ADM} relating the curvature and Nijenhuis tensor (with their derivatives) on an almost K\"ahler manifold; a version of this formula was used by Sekigawa in his original proof. An estimate based on this equality leads to our main result (Theorem~\ref{thm:main}),
proving a bound for the difference between the scalar curvature and the $*$-scalar curvature. This difference is zero in the cok\"ahler case; geometrically, this result shows that the underlying almost cosymplectic structure is in some sense close to being integrable. In particular, if one assumes the $*$-scalar curvature to be non-negative, 
then a compact, Einstein almost cok\"ahler manifold is necessarily cok\"ahler (Corollary \ref{cor:goldberg}). Also, as a consequence of Theorem~\ref{thm:main}, we recover
the result of \cite{CP}, namely, any compact, Einstein, almost cok\"ahler manifold whose Reeb vector field is Killing is cok\"ahler (Corollary~\ref{cor:reeb}). 

In section \ref{noncompact-Einstein}, we show that the odd-dimensional analogue of the Goldberg conjecture does not hold in the non-compact setting.
Using results of Lauret on Einstein solvmanifolds \cite{Lauret1, Lauret2}, we construct examples of 
non-compact, complete Einstein almost cok\"ahler
manifolds which are not cok\"ahler.

\section{Almost contact metric structures}\label{preliminaries}

We recall some definitions and results on almost contact metric manifolds (see
\cite{Blair1, Blair2, Boyer-Galicki} for more details).

Let $M$ be a $(2n + 1)$-dimensional manifold. An {\em almost contact structure} on $M$ consists of
a pair $(\alpha, \omega)$ of differential forms on $M$, where $\alpha$ is a 
$1$-form and $\omega$ is a 2-form, such that $\alpha\wedge \omega^n$ is a volume form. 
We call $(M, \alpha, \omega)$ an \textit{almost contact manifold}. 

Therefore, if $(\alpha, \omega)$ is an almost contact structure on $M$, 
the kernel of $\alpha$ defines a codimension one distribution ${\mathcal H} \,=\, \ker\alpha$, and 
the tangent bundle $TM$ of $M$ decomposes as
$$
TM = {\mathcal H} \oplus \langle \xi\rangle\, ,
$$
where $\xi$ is the nowhere vanishing vector field on $M$ (the {\em Reeb vector field} of 
$(\alpha, \omega)$) determined by the conditions
\begin{equation*}\label{Reeb}
\alpha(\xi)=1,\quad \iota_{\xi}(\omega)=0,
\end{equation*} 
where $\iota_{\xi}$ denotes the contraction by $\xi$. 

Since $\omega$ defines a non degenerate 2-form on $\mathcal H$, there exists an almost Hermitian 
structure $(J, g_{\mathcal H})$ on $\mathcal H$ with K\"ahler form the 2-form
$\omega$, that is, there are an endomorphism
$J \colon {\mathcal H} \longrightarrow {\mathcal H}$ and a metric $g_{\mathcal H}$ on ${\mathcal H}$
such that 
$$
J^2 = -Id_{\mathcal H}, \quad g_{\mathcal H}(X,Y)=g_{\mathcal H}(JX,JY), \quad \omega(X,Y) = g_ {\mathcal H} (J X,Y),
$$
for $X, Y\in\mathcal H$.

Thus, given an almost contact structure $(\alpha, \omega)$ on $M$ and fixed
an almost Hermitian structure $(J, g_{\mathcal H})$ on $\mathcal H$ with K\"ahler form 
$\omega$, we have the Riemannian metric $g$ on $M$ given by
$$
g=g_ {\mathcal H} + \alpha^2.
$$
In this case, we say that $g$ is a \dfn{compatible metric} with 
$(\alpha, \omega)$, and $(\alpha, \omega, g)$ is said to be an \dfn{almost contact metric structure} on $M$.
We call $(M, \alpha, \omega, g)$ an \dfn{almost contact metric manifold}.
(Notice that such a metric $g$ is not unique; indeed, it depends of the choice of 
$g_{\mathcal H}$.) 
Hence, for any point $p$ of $M$ there exist a 
neighborhood $U_p$ and an orthonormal coframe $\{e^1,\dotsc, e^{2n+1}\}$ with
\[\alpha=e^{2n+1}, \quad \omega=e^{12}+e^{34}+\dots + e^{2n-1,2n}.\]
Here and in the sequel, $e^{ij}$ is short for $e^i\wedge e^j$.

Under these conditions, the almost complex structure $J$ on $\mathcal H$ defines 
the endomorphism $\phi \colon TM \longrightarrow TM$ by
$$
\phi(X) = J (X), \quad  \phi(\xi)=0,
$$
for any $X\in\mathcal H$. One can check that the quadruplet $(\alpha, \xi, \phi,g)$ satisfies the conditions
\begin{equation*}
\label{almostcontactmetric}
\alpha(\xi) \,=\, 1, \quad \phi^2\,=\, - \id + \xi \otimes \alpha, \quad
g (\phi X, \phi Y)\,= \,g(X, Y) - \alpha(X) \alpha(Y)\, ,
\end{equation*}
for any vector fields $X, Y$ on $M$. Conversely, if $M$ is a differentiable manifold of dimension $2n+1$
with a quadruplet $(\alpha, \xi, \phi,g)$ satisfying \eqref{almostcontactmetric},
then $(\alpha, \omega, g)$ is an almost contact metric structure on $M$, where $\omega$ 
is the 2-form on $M$ given by
$$
\omega(X,Y)=g(\phi X, Y),
$$
for any vector fields $X, Y$ on $M$.

We say that an almost contact metric structure $(\alpha, \omega, g)$ on $M$
is \dfn{almost cok\"ahler} if $\alpha$ and $\omega$ are 
both closed, and \dfn{cok\"ahler} if they are both parallel under the Levi-Civita connection.
On an almost cok\"ahler manifold the forms $\alpha$ and $\omega$ are harmonic (see~\cite[Lemma 3]{GoldbergYano}),
and on a cok\"ahler manifold the Reeb vector field $\xi$ is Killing and parallel (see, for example~\cite{Blair1, Blair2}).

\section{Einstein almost cok\"ahler manifolds}\label{compact-Einstein}
In this section we consider almost cok\"ahler manifolds of dimension
$2n+1$ whose underlying metric $g$ is Einstein in the Riemannian sense, that is, the  Ricci curvature tensor 
satisfies 
\[\ric=\tau g,\]
where $\tau$ is a constant; the scalar curvature is then given by
\[s=(2n+1)\tau.\]
We do not assume compactness in this section.

From now on, we denote by $\nabla$ the Levi-Civita connection of $g$, which induces a second operator
\[\nabla^*\colon \Gamma(T^*M\otimes \Lambda^pM)\to \Gamma(\Lambda^pM), \quad \nabla^*=-\Tr \nabla.\]
If $e_1,\dotsc, e_{2n+1}$ denotes a local orthonormal frame and $e^1,\dotsc, e^{2n+1}$ is its dual coframe, we can express $\nabla^*$ by 
\[\nabla^*(e^i\otimes \beta) = -\sum_{j=1}^{2n+1}\langle e_j,\nabla_{e_j}e_i\rangle \beta - \nabla_{e_i}\beta.\]
Here and in the sequel, $\langle X,Y\rangle$ is an alternative notation for $g(X,Y)$.

The operator $\nabla^*$ is the formal adjoint of $\nabla$ in the sense that, when $\alpha$ and $\beta$ are compactly supported,
\[\int_M \langle \nabla \alpha,\beta\rangle = \int_M \langle \alpha, \nabla^*\beta\rangle.\]
Moreover, when $\beta=\nabla\alpha$ the equation holds pointwise, i.e.
\[\abs{\nabla\alpha}^2=\langle\alpha, \nabla^*\nabla\alpha\rangle.\]
We denote by $R$ the curvature tensor given by 
\[ R(X,Y)Z = \nabla_X\nabla_Y Z - \nabla_Y\nabla_X Z - \nabla_{[X,Y]}Z;\]
we note that  \cite{ADM} uses the opposite sign. Recall the classical formula due to Weitzenb\"ock (see e.g. \cite{Besse}): given a $p$-form $\eta$,
\begin{equation}
 \label{eqn:Weitzenbock}
\Delta\eta=-\sum_{h,k}e^h\wedge (e_k\hook R(e_h,e_k)\eta)+ \nabla^*\nabla \eta.
\end{equation}
On an almost contact metric manifold, the $*$-Ricci tensor is defined as
\[\ric^*(X,Y)=\omega\bigl(\sum_{i=1}^{2n} R(X,e_i)(Je_i), Y\bigr).\]
We shall also consider the $*$-Ricci form
\[\rho^*(X,Y)=\sum_i \langle R(X,e_i)(Je_i), Y\rangle\]
and set
\[\tau^*=\frac1n \langle \omega,\rho^*\rangle .\]
\begin{lemma}
\label{lemma:nablastarnabla}
On any Einstein almost cok\"ahler manifold $(M,\alpha,\omega,g)$ with 
$\ric=\tau g$, 
\[\nabla^*\nabla\alpha = -\tau\alpha, \quad \nabla^*\nabla\omega = 2(\rho^*-\tau\omega).\]
\end{lemma}
\begin{proof}
If $\eta$ is a $1$-form, the  Weitzenb\"ock formula \eqref{eqn:Weitzenbock} specializes to 
\[\Delta\eta = \nabla^*\nabla\eta +\ric(\eta).\]
where $\ric$ denotes the Ricci operator. By \cite[Lemma 3]{GoldbergYano}, $\alpha$ is harmonic. Then, using the Einstein condition
$\ric=\tau g$, 
we obtain the first formula.

If $\eta$ is a $2$-form,  \eqref{eqn:Weitzenbock} can be written as 
\[\Delta\eta =\nabla^*\nabla\eta  + \ric(\eta)+2\tilde R(\eta);\]
where the Ricci operator acts as derivations, and $\tilde R(\eta)$ denotes the image of $\eta$ under the curvature operator $\tilde R\in\Gamma(\End(\Lambda^2T^*M))$. Applying this to $\eta=\omega$,
\[2\tilde R(\omega)= \sum_i R(e_i,Je_i);\]
by the Bianchi identity, we find
\begin{multline*}
2R(\omega)(X,Y)=\sum_i R(e_i,Je_i,X,Y) = \sum_i -R(Je_i,X,e_i,Y)-R(X,e_i,Je_i,Y)\\
=-2\sum_i R(X,e_i,Je_i,Y)=-2\rho^*(X,Y). 
\end{multline*}
Hence the Weitzenb\"ock formula  gives
\[\nabla^*\nabla\omega = -\ric(\omega)-2R(\omega)=-2\tau\omega+2\rho^*,\]
where we have used the
facts that $\omega$ is harmonic (see \cite[Lemma 3]{GoldbergYano}) and 
the identity acts as twice the identity on $\Lambda^2T^*M$.
\end{proof}

Our first observation is that Einstein cok\"ahler manifolds, unlike their even-dimensional counterpart, are necessarily Ricci-flat. The proof exploits the existence of a non-zero harmonic one-form $\alpha$ and mimics Bochner's proof that a \emph{compact} Einstein manifold with positive curvature cannot have $b_1>0$ (see \cite{Bochner}).

\begin{proposition}
\label{prop:nocokahlereinstein}
Any  Einstein cok\"ahler manifold  $(M,\alpha,\omega,g)$ is Ricci-flat.
\end{proposition}
\begin{proof}
By hypothesis, $\nabla\alpha=0$, so Lemma~\ref{lemma:nablastarnabla} implies that the scalar curvature is zero.
\end{proof}

This result does not hold if $M$ is only assumed to be an almost cok\"ahler manifold; indeed, in Section~\ref{noncompact-Einstein} we will construct an almost cok\"ahler Einstein manifold with negative scalar curvature.  However,
we can prove directly from Lemma~\ref{lemma:nablastarnabla} that the scalar curvature is not allowed to be positive:
\begin{proposition}\label{prop:nonpositivecurvature}
Let $(M,\alpha,\omega,g)$ be a an Einstein almost cok\"ahler manifold. Then
\begin{equation}
 \label{eqn:tautaustarineq}
0\leq-\tau\leq  2n(\tau^*-\tau).
\end{equation}
\end{proposition}
\begin{proof}
By Lemma~\ref{lemma:nablastarnabla},
\[\abs{\nabla\alpha}^2=-\tau, \quad \abs{\nabla\omega}^2=2n(\tau^*-\tau).\]
Observe that 
$*\alpha=\frac1{n!}\omega^n$; therefore, for any tangent vector $X$,
\begin{multline*}
*\nabla_X\alpha =\nabla_X*\alpha =\nabla_X \frac1{n!}\omega^n=\nabla_X\omega \wedge\frac1{(n-1)!}\omega^{n-1}\\
= \sum_{i=1}^{n} (\nabla_X\omega)(e_{2i-1},e_{2i})e^{1,\dots, 2n} + \nabla_X\omega(e_{2i},\xi)e^{1,\dots,\widetilde{2i-1},\dots,2n+1}\\
+ \nabla_X\omega(e_{2i-1},\xi)e^{1,\dots,\widetilde{2i},\dots,2n+1};
\end{multline*}
it follows that $\abs{*\nabla_X\alpha}^2\leq \abs{\nabla_X\omega}^2$, and consequently $0\leq \abs{\nabla\alpha}^2\leq \abs{\nabla\omega}^2$; the statement follows.

\end{proof}

\section{The compact case}\label{compact:case}
In this section we consider potential counterexamples of the Goldberg conjecture, namely compact Einstein manifolds with either an almost K\"ahler structure that is not K\"ahler or an almost cok\"ahler structure that is not cok\"ahler, and prove an integral bound on the difference between scalar curvature and $*$-scalar curvature. The main ingredient is a formula of \cite{ADM} that relates the curvature on an almost K\"ahler manifold to the covariant derivative of the fundamental form.

In order to introduce this formula, let $(N,h,J,\Omega)$ be an almost K\"ahler manifold with Riemannian metric $h$, almost complex structure $J$
and K\"ahler form $\Omega$, and let $\nabla$ be the Levi-Civita connection. Borrowing notation 
from \cite{ADM}, we decompose the Ricci tensor in two components
\[\ric'\in [S^{1,1}], \quad \ric''\in[\![S^{2,0}]\!];\]
here, $[S^{1,1}]$ represents the real subspace of conjugation-invariant elements of $S^{1,1}$, and $[\![S^{2,0}]\!]$ represents $[S^{2,0}+S^{0,2}]$.
In other words, $\ric'$ is the component that commutes with $J$, and $\ric''$ is the component that anticommutes with $J$. 
We define the Ricci and *Ricci forms as
\[\rho^h (X,Y)=\ric'(JX,Y), \quad \rho^{*h}(X,Y)=\sum_{i=1}^{2n} R(X,e_i,Je_i,Y),\]
where $\{e_1,\dots, e_{2n}\}$ is a local orthonormal frame. Note that
in the notation of \cite{ADM}, we can write $\rho^{*h}=-R(\omega)$, where the different sign follows from the conventions.

The scalar and  $*$-scalar curvatures are defined by
\[s= 2\langle \rho^h, \Omega\rangle, \quad s^*=2\langle \rho^{*h},\Omega\rangle.\]
The Weitzenb\"ock formula (see e.g. \cite{ADM}) gives $\nabla^*\nabla\Omega=2(\rho^{*h}-\rho^h)$; in particular,
\begin{equation}
 \label{eqn:intsminusstar}
\abs{\nabla\Omega}^2 = s^*-s.
\end{equation}

The curvature tensor $R$ takes values in 
\[S^2([\![\Lambda^{2,0}]\!]+[\Lambda^{1,1}])=S^2([\![\Lambda^{2,0}]\!]) + [\![\Lambda^{2,0}]\!]\otimes [\Lambda^{1,1}] + S^2([\Lambda^{1,1}]);\]
we denote by $\tilde R$ the first component in this decomposition. As an endomorphism of $[\![\Lambda^{2,0}]\!]$, $\tilde R$ decomposes in two components that commute (respectively, anticommute) with $J$, namely
\[\tilde R=\tilde R'+\tilde R''.\]
We also introduce the two-form 
\[\phi(X,Y)=\langle \nabla_{JX}\Omega,\nabla_Y\Omega\rangle;\]
this is well defined and $J$-invariant by the following observation, which is implicit in \cite{ADM}:
\begin{lemma}
On an almost-K\"ahler manifold $(N,h,J,\Omega)$,
\begin{equation}
 \label{eqn:nablaomegainv}
\langle \nabla_{X}\Omega,\nabla_Y\Omega\rangle= \langle \nabla_{JX}\Omega,\nabla_{JY}\Omega\rangle .
\end{equation}
\end{lemma}
\begin{proof}
The image of the infinitesimal action of $\so(2n)$ on $\Omega$ is $[\![\Lambda^{2,0}]\!]$; therefore, the covariant derivative $\nabla\Omega$ lies in 
\[\Lambda^1\otimes [\![\Lambda^{2,0}]\!]= [\![\Lambda^{1,0}\otimes \Lambda^{2,0}]\!]+[\![\Lambda^{1,0}\otimes \Lambda^{0,2}]\!]\]
(see also \cite[Lemma 3.3]{Salamon}). Since the Levi-Civita connection is torsion-free, $d\Omega$ is the image of $\nabla\Omega$ under the skew-symmetrization map
\[\Lambda^1\otimes [\![\Lambda^{2,0}]\!]\to \Lambda^3, \quad \alpha\otimes\beta\mapsto\alpha\wedge\beta;\]
thus, $\nabla\Omega$ is in the kernel of this map, which has the form 
\begin{equation}
\label{eqn:Vtheunmodule}
[\![V]\!]\subset[\![\Lambda^{1,0}\otimes \Lambda^{2,0}]\!].
\end{equation}
Fixing a basis  $\{\omega_a\}$ on $\Lambda^{2,0}$, orthonormal for the standard hermitian product,
the inclusion~\eqref{eqn:Vtheunmodule} implies that $\nabla\Omega$ can be written as
\[\nabla\Omega = \sum_a \lambda_a\otimes \omega_a + \overline{\lambda_a}\otimes\overline{\omega_a} \in \Lambda^{1,0}\otimes \Lambda^{2,0}+\Lambda^{0,1}\otimes \Lambda^{0,2};\]
it follows that
\[\langle \nabla\Omega, \nabla\Omega\rangle = \sum_a \lambda_a\otimes \overline{\lambda_a}+\overline{\lambda_a}\otimes \lambda_a\]
lies in $[S^{1,1}]$.
\end{proof}

\begin{proposition}[Apostolov-Dr\u{a}ghici-Moroianu \cite{ADM}]
\label{prop:ApostolovDraghiciMoroianu}
On  an almost K\"ahler manifold $(N,h,J,\Omega)$, there is a one-form $\gamma$ such that 
\[\Delta(s-s^*)+d^*\gamma + 2\abs{\ric''}^2 - 8 \abs{\tilde R''}^2 - \abs{\nabla^*\nabla\Omega}^2-\abs{\phi}^2+4\langle \rho^h,\phi\rangle - 4\langle\rho^h,\nabla^*\nabla\Omega\rangle=0.\]
\end{proposition}

Notice that this formula holds locally, and compactness is not assumed. On the other hand, integrating this identity on a compact manifold yields a formula where the first two terms do not appear,  since the codifferential of a one-form is always the Hodge dual of an exact form.

If $N$ is also Einstein and compact, we can derive from this formula an integral bound on the difference $s^*-s$; by \eqref{eqn:intsminusstar}, this means that $N$ is close to being K\"ahler. More precisely: 
\begin{theorem}
\label{thm:evenpinch}
Every compact Einstein almost K\"ahler manifold $(N,h,J,\Omega)$  which is not K\"ahler satisfies
\[s<\frac1V\int s^*\leq \frac15s<0,\]
where $V$ denotes the volume.
\end{theorem}
\begin{proof}
Let the dimension of $N$ be $2n$. 
The Einstein condition implies $\ric=\frac1{2n} s\id$, so $\ric''$ is identically zero and $\rho^h=\frac1{2n}s\Omega$; integrating the formula of Proposition~\ref{prop:ApostolovDraghiciMoroianu}, we obtain
\[\int - 8 \abs{\tilde R''}^2 - \abs{\nabla^*\nabla\Omega}^2-\abs{\phi}^2+\frac2n s\langle \Omega,\phi\rangle - \frac2n s\langle\Omega,\nabla^*\nabla\Omega\rangle=0.\]
The Weitzenb\"ock formula gives 
\begin{multline*}
 \int \abs{\nabla^*\nabla\Omega}^2=\int\abs{2(\rho^{*h}-\frac1{2n}s\Omega)}^2 \geq 
 \int\langle  2\rho^{*h}-\frac1{n}s\Omega,\Omega\rangle^2 \frac1{\abs{\Omega}^2}\\
=\int \frac1{n}(s^*-s)^2,
\end{multline*} 
where we have applied the Cauchy-Schwartz inequality at each point and used $\abs{\Omega}^2=n$.
With respect to an orthonormal basis $\{e_1,\dots, e_{2n}\}$, 
\begin{equation}
\label{eqn:exp1}
\langle \Omega,\phi\rangle = \frac12\sum \phi(e_i,Je_i)=\frac12\sum \langle \nabla_{Je_i}\Omega,\nabla_{Je_i}\Omega\rangle = \frac12\abs{\nabla\Omega}^2.
\end{equation}
Therefore
\[
\int\frac2n s\langle \Omega,\phi\rangle - \frac2n s\langle\Omega,\nabla^*\nabla\Omega\rangle=
-\frac1ns \int \langle \Omega,\nabla^*\nabla\Omega\rangle 
.\]
Hence
\[\int \abs{\phi}^2 	\leq \int  -\frac1{n}(s^*-s)^2 -\frac1{n}s(s^*-s) = -\frac1{n}\int s^* (s^*-s);\]
by \eqref{eqn:intsminusstar}, this is only possible if $s<0$, consistently with Sekigawa's result \cite{Sekigawa}. Again by \eqref{eqn:intsminusstar}, we can write
\[s^*-s=-fs, \quad  f>0.\]
On the other hand,
\begin{multline*}
\abs{\phi}^2 = \frac12\sum_{1\leq i,j\leq 2n} \phi(Je_i,e_j)^2=\frac12\sum_{1\leq i,j\leq 2n}\langle \nabla_{e_i}\Omega, \nabla_{e_j}\Omega\rangle^2\\
\geq\sum_{i} \frac12\abs{\nabla_{e_i}\Omega}^4 \geq \frac1{4n}\bigl(\sum_{i}\abs{\nabla_{e_i}\Omega}^2\bigr)^2 ,
\end{multline*}
where we have used the generalized mean inequality. 
Summing up,
\[\frac1{4n}s^2\int f^2 = \frac1{4n}\int\abs{\nabla\Omega}^4\leq \int \abs{\phi}^2\leq  \int\frac1n (1-f)fs^2.\]
By the Cauchy-Schwartz inequality, this gives
\[\frac5{4V}(\int f)^2\leq \frac54 \int f^2 \leq \int f.\]
In particular, $\int f\leq \frac45 V$.
\end{proof}
As a consequence, we obtain a variation of Sekigawa's theorem that applies when $s^*$, as opposed to $s$, is non-negative.
\begin{corollary}
\label{cor:evengoldberg}
Let  $(N,h,J,\Omega)$ be a compact, Einstein, almost K\"ahler manifold. If $\int s^*\geq0$,  then  $(M,g,\alpha,\omega)$ is K\"ahler.
\end{corollary}

One way of approaching the odd-dimensional case is through the following observation:
\begin{proposition}
\label{prop:product}
Let  $(M, g, \omega,\alpha)$ be an  almost cok\"ahler Einstein manifold of dimension $2n+1$. A natural almost K\"ahler structure is induced on $M\times M$; it is Einstein and satisfies
\[s=(4n+2)\tau, \quad s^*=4n\tau^*.\]
\end{proposition}
\begin{proof}
Let $(\tilde M,\tilde g, \tilde\omega,\tilde\alpha)$ be another copy of $(M, g, \omega,\alpha)$, and consider the Riemannian product $N=M\times\tilde M$ with the almost-K\"ahler structure determined by
\[\Omega = \omega + \tilde\omega + \alpha\wedge\tilde\alpha.\]
The Ricci tensor on $N$ is given by $\tau g + \tau\tilde g$, giving $s=(4n+2)\tau$; the formula for $s^*$ can be derived similarly, or from 
\[s^*-s=\abs{\nabla\Omega}^2 = \abs{\nabla\omega}^2 +\abs{\nabla\tilde\omega}^2 +\abs{\nabla\alpha}^2 +\abs{\nabla\tilde\alpha}^2 = 4n(\tau^*-\tau)-2\tau.\qedhere\]
\end{proof}

\begin{corollary}
\label{cor:nonoptimal}
Let  $(M,g,\alpha,\omega)$ be a compact, Einstein, almost cok\"ahler manifold of volume $V$ and dimension $2n+1$.  Then 
\[\tau \leq\frac1V\int \tau^*\leq \frac15\tau\leq0.\]
\end{corollary}
\begin{proof}
By Proposition~\ref{prop:nonpositivecurvature},
 $\tau^*\geq \tau$ and $\tau\leq0$. Let $N=M\times M$ with the induced almost K\"ahler structure, as in Proposition~\ref{prop:product}. If $N$ is K\"ahler, then 
\[(4n+2)\tau =s=s^*=4n\tau^*.\]
This is only possible if $\tau=\tau^*=0$, which makes the statement hold trivially.
If $N$ is not K\"ahler, Theorem~\ref{thm:evenpinch} implies that
\[\frac1V\int (4n)\tau^*\leq \frac15(4n+2)\tau.\qedhere\]
\end{proof}

The estimate of Corollary~\ref{cor:nonoptimal} only makes use of the fact that the induced almost K\"ahler structure on $M\times M$ is Einstein, neglecting other conditions that follow from $M$ being almost cok\"ahler. We can obtain a sharper estimate by making use of these conditions; in order to simplify the argument, we shall work with $M\times S^1$ rather than $M\times M$.

\begin{theorem}
\label{thm:main}
Let  $(M,g,\alpha,\omega)$ be a compact, Einstein, almost cok\"ahler manifold of volume $V$ and dimension $2n+1$.  Then either 
\begin{enumerate}
 \item[1)] $\tau=0=\tau^*$ and  $(M,g,\alpha,\omega)$ is cok\"ahler; or
 \item[2)] \[\frac1{2n}\leq \frac1V\int \frac{\tau-\tau^*}{\tau} \leq
 \frac{4n-1+\sqrt{16n^2-8n-14}}{10n}\]
 and  $(M,g,\alpha,\omega)$ is not cok\"ahler.
\end{enumerate}
\end{theorem}

\begin{proof}
Let $t$ be a coordinate on $S^1=\{e^{it}\}$, and write $\theta=dt$. On the product $M\times S^1$,
fix the product metric $h=g+\theta\otimes\theta$, and set
\[\Omega=\omega+\alpha\wedge\theta;\]
we thus obtain a compact almost-K\"ahler manifold $(M\times S^1,h,\Omega,J)$.

By construction, the Ricci tensor of $h$ is 
\[\ric^h=\ric^g = \tau g,\]
where $\tau$ is a constant; the scalar curvature of $h$ is then $(2n+1)\tau$, and  $\ric^h$ splits into the two components 
\[\ric'= \tau(h-\frac12\alpha\otimes\alpha-\frac12\theta\otimes\theta), \quad \ric''=\frac12\tau(\alpha\otimes\alpha-\theta\otimes\theta).\]
By definition, the Ricci form is
\[\rho^h = \tau(\Omega-\frac12\alpha\wedge\theta).\]
Since the vanishing of $\nabla_{\D t}\Omega$ implies that $\langle\phi,\alpha\wedge\theta\rangle$ is zero,  using \eqref{eqn:exp1} we conclude that
\[\langle\rho^h,\phi\rangle= \frac12\tau\abs{\nabla\Omega}^2.\]

By construction,
\[\nabla\Omega=\nabla\omega + \nabla\alpha\wedge\theta,\]
giving
\[\abs{\nabla\Omega}^2=\abs{\nabla\omega}^2 + \abs{\nabla\alpha}^2.\]
Integrating over $M\times S^1$, we find
\begin{equation}
 \label{eqn:ktau}
\int \langle \nabla^*\nabla\omega,\omega\rangle= \int 2n(\tau^*-\tau), \quad \int  \langle\nabla^*\nabla\alpha,\alpha\rangle =\int -\tau,
\end{equation}
where we have used Lemma~\ref{lemma:nablastarnabla}.

Similarly, 
\[\langle \nabla\rho^h,\nabla\Omega\rangle = \tau\abs{\nabla\Omega}^2-\frac12\tau\langle(\nabla\alpha)\wedge\theta,\nabla\Omega\rangle
= \tau\abs{\nabla\Omega}^2-\frac12\tau\abs{\nabla\alpha}^2.\]
Finally, observe that
\[\nabla^*\nabla\Omega=\nabla^*(\nabla\omega + \nabla\alpha\wedge\theta)=\nabla^*\nabla\omega + \nabla^*\nabla\alpha\wedge\theta =2 (\rho^*-\tau\omega)-\tau\alpha\wedge\theta.\]
where $\rho^*$ is the odd-dimensional *Ricci. We can write
\[\rho^*=\tau^*\omega+\rho^*_0,\quad\langle \rho^*_0,\omega\rangle=0,\]
giving
\[\abs{\nabla^*\nabla\Omega}^2=4n(\tau^*-\tau)^2 +\tau^2 + 4\abs{\rho^*_0}^2.\]

We can decompose the space $[\![\Lambda^{2,0}]\!]$ as
\[[\![\Lambda^{2,0}]\!]_{\mathcal{H}} \oplus[\![\Lambda^{2,0}]\!]_{\alpha} = \Span{e^{ij}-Je^i\wedge Je^j}\oplus \Span{\alpha\wedge e^i-\theta\wedge Je^i};\]
writing the curvature as
\[R=\sum a_{ijkl}e^{ij}\otimes e^{kl} +b_{ijk}e^{ij}\odot \alpha\wedge e^k+c_{ij}\alpha\wedge e^i\otimes\alpha\wedge e^j,\]
its projection on $S^2([\![\Lambda^{2,0}]\!]_{\alpha})$ is 
\[
\frac14c_{ij}(\alpha\wedge e^i-\theta\wedge Je^i)\otimes(\alpha\wedge e^j-\theta\wedge Je^j).\]
If we further project on the component that commutes with $J$, we obtain 
\[
 \frac18c_{ij}\bigl((\alpha\wedge e^i-\theta\wedge Je^i)\otimes(\alpha\wedge e^j-\theta\wedge Je^j)
+ (\alpha\wedge Je^i+\theta\wedge e^i)\otimes(\alpha\wedge Je^j+\theta\wedge e^j)\bigr);\]
taking norms, we find
\[ \abs{\tilde R''}^2 \geq \frac1{8}\sum_{i,j}c_{ij}^2 \geq \frac18\sum c_{ii}^2\geq\frac1{16n} \bigl(\sum c_{ii}\bigr)^2= \frac{1}{16n}\tau^2.\]

Integrating the formula of Proposition~\ref{prop:ApostolovDraghiciMoroianu}, we can now compute
\begin{multline*}
0 = \int ( 2\abs{\ric''}^2 - 8 \abs{\tilde R''}^2 - \abs{\nabla^*\nabla\Omega}^2-\abs{\phi}^2+4\langle \rho^h,\phi\rangle - 4\langle\rho^h,\nabla^*\nabla\Omega\rangle)\\
 = \int ( \tau^2 - 8 \abs{\tilde R''}^2 -4n(\tau-\tau^*)^2 -\tau^2 - 4\abs{\rho^*_0}^2  -\abs{\phi}^2+2\tau\abs{\nabla\Omega}^2 - 4 \tau\abs{\nabla\Omega}^2+2\tau\abs{\nabla\alpha}^2 )\\
 = \int (  - 8 \abs{\tilde R''}^2 -4n(\tau-\tau^*)^2 -4\abs{\rho^*_0}^2-\abs{\phi}^2- 4 n\tau(\tau^*-\tau) ).
 \end{multline*}

Summing up,
\[\int \abs{\phi}^2\leq \int -4n\tau^*(\tau^*-\tau)-\frac1{2n}\tau^2.\]
If $\tau=0$, this implies that $\int (\tau^*)^2$ is non-positive, hence $\tau^*=0$. By \eqref{eqn:ktau}, this is only possible if the structure is cok\"ahler, 
giving the first case in the statement. 

Assume now that $\tau<0$.
Observe that $\phi(\xi,Y)=0$, because \eqref{eqn:nablaomegainv} implies
\[
\norm{ \nabla_{X}\Omega}=\norm{\nabla_{JX}\Omega}.
\]
By construction,
\begin{multline*}
\abs{\phi}^2 = \frac12\sum_{1\leq i,j\leq 2n} \phi(Je_i,e_j)^2
=\frac12\sum_{1\leq i,j\leq 2n}
 \bigl(\langle \nabla_{e_i} \omega, \nabla_{e_j}  \omega\rangle+\langle \nabla_{e_i}  \alpha, \nabla_{e_j}\alpha\rangle\bigr)  \bigr)^2\\
 \geq\frac12\sum_{1\leq i\leq 2n} \bigl(\abs{\nabla_{e_i} \omega}^2+\abs{\nabla_{e_i} \alpha}^2\bigr)^2
 \geq\frac1{4n}\biggl(\sum_{1\leq i\leq 2n} \bigl(\abs{\nabla_{e_i} \omega}^2+\abs{\nabla_{e_i} \alpha}^2\bigr)\biggr)^2\\
 =\frac1{4n}(2n(\tau^*-\tau)-\tau)^2.
\end{multline*}
It follows that
\[ \frac1{4n}\int (2n(\tau^*-\tau)-\tau)^2\leq\int \bigl( -4n(\tau^*-\tau)^2 -4n\tau(\tau^*-\tau)-\frac1{2n}\tau^2\bigr);\]
by the Cauchy-Schwartz  inequality, 
\[\frac1V 5n\bigl(\int \tau^*-\tau\bigr)^2 \leq \int 5n(\tau^*-\tau)^2\leq \int \bigl(-(4n-1)\tau(\tau^*-\tau)-\frac3{4n}\tau^2\bigr).\]
Since $\tau$ is a constant, this is a second degree inequality in the variable \mbox{$\int (\tau^*-\tau)$}  with constant coefficients; solving explicitly, we find
\[\frac1V \int (\tau^*-\tau)\leq  \frac{4n-1+\sqrt{16n^2-8n-14}}{10n}(-\tau).\]

The remaining part of the statement follows from Proposition~\ref{prop:nonpositivecurvature}.
\end{proof}

As an immediate consequence, we find:
\begin{corollary}
\label{cor:goldberg}
Let  $(M,g,\alpha,\omega)$ be a compact, Einstein, almost cok\"ahler manifold. If either $\tau\geq0$ or $\int \tau^*\geq0$,  then  $(M,g,\alpha,\omega)$ is Ricci-flat and cok\"ahler.
\end{corollary}

In particular, in the case $\tau=0$ we recover the following result of \cite{CP}: 
\begin{corollary}[\cite{CP}]
\label{cor:reeb}
If $(M,g,\alpha,\omega)$ is a compact, Einstein, almost cok\"ahler manifold on which the Reeb vector field is Killing, then
$(M,g,\alpha,\omega)$ is Ricci-flat and cok\"ahler.
\end{corollary}
\begin{proof}
The condition on the Reeb vector field implies that $\nabla\alpha$ is skew-symmetric, and therefore completely determined by $d\alpha$.
Since $\alpha$ is closed, it is also parallel. This implies that $\tau=0$, so Corollary~\ref{cor:goldberg} applies.
\end{proof}

\section{Einstein almost cok\"ahler manifolds which are not cok\"ahler}\label{noncompact-Einstein}
In this section we give a five-dimensional example of an Einstein almost cok\"ahler manifold which is not cok\"ahler.

We consider a standard extension of a $4$-dimensional Ricci nilsoliton (see \cite{Heber}), namely the Lie algebra
$\lie{g}=\langle e_1, e_2, e_3, e_4, e_5\rangle$ defined by the equations
$$
de^1=\frac{\sqrt{3}}{2} e^{25}+\frac{1}{2} e^{14}, \quad de^2=\frac{\sqrt{3}}{2}  e^{15}+\frac{1}{2} e^{24}, \quad 
de^3= e^{12}+e^{34}, \quad de^4= de^5=0,
$$
where $\langle e^1, e^2, e^3, e^4, e^5\rangle$ is the dual basis for $\lie{g}^*$ and $e^{ij}$ is short for $e^i\wedge e^j$.

We define $G$ to be the connected, simply connected Lie group with Lie algebra $\lie{g}$.
\begin{proposition}
\label{prop:example}
The solvable Lie group $G$ has an Einstein almost cok\"ahler structure which is not cok\"ahler and satisfies 
\[\frac{\tau-\tau^*}{\tau}=\frac14.\]
\end{proposition}
\begin{proof}
Let $g$ be the left invariant metric on $G$ given by
\[g= (e^1)^2+(e^2)^2+(e^3)^2+(e^4)^2+(e^5)^2.\]
One can check that $g$ is an Einstein metric on $G$. Indeed, the Ricci curvature tensor satisfies 
\[\ric=-\frac{3}{2} e^1\otimes e^1-\frac{3}{2} e^2\otimes e^2-\frac{3}{2} e^3\otimes e^3-\frac{3}{2} e^4\otimes e^4-\frac{3}{2} e^5\otimes e^5.\]
Take the pair $(\alpha, \omega)$ of forms on $G$ given by
\[\alpha=e^5, \quad \omega=e^{12}+e^{34}.\]
Then $(\alpha, \omega, g)$ defines an almost cok\"ahler structure on $G$ since $d\alpha=d\omega=0$,
$\alpha\wedge\omega^2\not=0$ and $g$ is compatible with $(\alpha, \omega)$ in the sense given in section \ref{preliminaries}.
Moreover $\rho^*=-\frac34 e^{12}-\frac32 e^{34}$, so
\[\tau = -\frac32,\quad  \tau^*=-\frac98.\]
Since $\tau^*-\tau$ is not zero,  $(\alpha, \omega, g)$ is not a cok\"ahler structure. 
In fact, there is no parallel left invariant
$2$-form on this Lie group, so no invariant cok\"ahler structure compatible with the metric $g$ exists.
\end{proof}
\begin{remark}
Even though this example is not compact, the value of $\frac{\tau-\tau^*}{\tau}$ is consistent with the inequalities of Theorem~\ref{thm:main}. In fact, it is the smallest value compatible with \eqref{eqn:tautaustarineq}.
\end{remark}
\begin{remark}
We note that a result in \cite{Dotti} asserts that no solvable unimodular Lie group
admits a left invariant metric of strictly negative Ricci curvature. In fact, it is easy to verify that the Lie group $G$ of Proposition~\ref{prop:example} is not unimodular; in particular, it does not have a uniform discrete subgroup, i.e. a discrete subgroup $\Gamma$ such that $\Gamma\backslash G$ is compact.
\end{remark}

\bigskip

\vskip.3cm

\noindent {\bf Acknowledgments.} 
We are very grateful to the referees for useful comments
that helped to improve the paper.
This work was partially supported through Project MICINN (Spain) 
MTM2011-28326-C02-02,  Project MINECO (Spain) MTM2014-54804-P, FIRB 2012 ``Geometria differenziale e teoria geometrica delle funzioni'' and PRIN 2010-2011 ``Variet\`a reali e complesse: geometria, topologia e analisi armonica''.

\small\noindent Dipartimento di Matematica e Applicazioni, Universit\`a di Milano Bicocca,  Via Cozzi 55, 20125 Milano, Italy.\\
\texttt{diego.conti@unimib.it}

\medskip
\small\noindent Universidad del Pa\'{\i}s Vasco, Facultad de Ciencia y Tecnolog\'{\i}a, Departamento de Matem\'aticas,
Apartado 644, 48080 Bilbao, Spain. \\
\texttt{marisa.fernandez@ehu.es}\\

\end{document}